\documentclass{amsart}

\newtheorem{thm}{Theorem}[section]
\newtheorem{cor}[thm]{Corollary}
\newtheorem{lem}[thm]{Lemma}

\newtheorem{prop}[thm]{Proposition}

\newtheorem{defi}[thm]{Definition}

\theoremstyle{definition}
\newtheorem{rem}[thm]{Remark}
\newtheorem*{notaetoile}{Notation}

\usepackage{amsmath,amsfonts,amssymb,amsthm,latexsym,bbm,color}

\usepackage{multirow}

\newcommand{\N}{\mathbb N}

\newcommand{\C}{\mathbb C}
\newcommand{\R}{\mathbb R}

\renewcommand{\P}{\mathbb P}
\newcommand{\E}{\mathbb E}

\newcommand{\ie}{{\it i.e. }}

\newcommand{\la}{\langle}
\newcommand{\ra}{\rangle}

\DeclareMathOperator{\id}{id}
\DeclareMathOperator{\tr}{tr}
\DeclareMathOperator{\Tr}{Tr}

\DeclareMathOperator{\Ad}{Ad}

\begin{document}

\title[Quantum random walks and minors of Hermitian Brownian Motion]{Quantum random walks and minors of Hermitian Brownian motion}

\author{Fran\c cois Chapon} 
\address{Laboratoire de probabilit\'es et Mod\`eles Al\'eatoires, Universit\'e Paris 6, 4 place Jussieu, 75252, Paris Cedex 05}
\email{francois.chapon@upmc.fr}
\author{Manon Defosseux}
\address{Laboratoire de Math\'ematiques Appliqu\'ees \`a Paris 5, Universit\'e Paris 5, 45 rue des  Saints P\`eres, 75270 Paris Cedex 06.}
\email{manon.defosseux@parisdescartes.fr}

\begin{abstract} 

Considering quantum random walks, we construct discrete-time approximations of the eigenvalues processes of minors of Hermitian Brownian motion. It has been recently proved  by Adler, Nordenstam and van Moerbeke in \cite{Nordenstam} that the process of eigenvalues of two consecutive minors of an Hermitian Brownian motion is a Markov process, whereas if one considers more than two consecutive minors, the Markov property fails. We show that there are analog results   in the  noncommutative counterpart and establish the Markov property of eigenvalues of some particular submatrices of Hermitian Brownian motion.

\end{abstract}
\maketitle

\section{Introduction}
Let $(M(t),t\ge 0)$ be a $2\times 2$  Hermitian  Brownian motion with null trace, \ie  \[
M(t)=\left[
\begin{array}{cc}
  B_{1}(t) & B_2(t)+iB_3(t)  \\
  B_2(t)-iB_3(t)& -B_{1}(t)
\end{array}
\right],\, t\ge 0,
\]
 where $(B_1,B_2,B_3)$ is a standard Brownian motion in $\R^3$.  It\^o's calculus easily shows that the process \begin{align}\label{d=2} \big(B_1(t),\sqrt{B_{1}^2(t)+B_{2}^2(t)+B_{3}^2(t)}\big), t\ge 0,\end{align} is a Markovian process on $\R^2$. 
 Let us recall how noncommutative discrete-time approximation of this process can be constructed, following \cite{bianemeyer}.
  For this, we consider the set $\operatorname{M}_{2}(\C)$ of $2\times 2$ complex matrices, endowed with the state \[ \tr(M)=\frac{1}{2}\Tr(M), \, M\in \operatorname{M}_2(\C),\] and  the Pauli matrices
\[
x= 
\begin{pmatrix}
 0 & 1 \\
 1 & 0  
\end{pmatrix}
, \quad y=
\begin{pmatrix}
 0 & -i \\
 i & 0  
\end{pmatrix}
,\quad z=
\begin{pmatrix}
 1 & 0 \\
 0 & -1  
\end{pmatrix},
\]
which satisfy the commutation relations 
\[
[x,y]=2iz,\, [y,z]=2ix,\textrm{ and } [z,x]=2iy.
\]
The matrices $x,y$ and $z$  define three noncommutative Bernoulli variables. Consider the algebra  $\operatorname{M}_2(\C)^{\otimes \infty},$  endowed with the infinite product state, still denoted  $\tr$, defined by
\[
\tr(a_1\otimes\dots\otimes a_n\otimes I^{\otimes\infty})=\tr(a_1)\cdots\tr(a_n),\quad \text{ for } a_1,\dots, a_n\in \operatorname{M}_2(\C),
\]
where $I$ is the identity matrix of $\operatorname{M}_2(\C)$.
 Define, for all $i\in\N^*$, the elements 
\[
x_i=I^{\otimes (i-1)}\otimes x\otimes I^{\otimes\infty} , \ y_i=I^{\otimes (i-1)}\otimes y\otimes I^{\otimes\infty} , \ z_i=I^{\otimes( i-1)}\otimes z\otimes I^{\otimes\infty},
\]
as well as the partial sums
\[
X_n=\sum_{i=1}^nx_i, \ Y_n=\sum_{i=1}^n y_i , \ Z_n=\sum_{i=1}^n z_i,\quad n\ge 1.
\]
The processes $(X_n)_{n\ge 1}$, $(Y_n)_{n\ge 1}$ and $(Z_n)_{n\ge 1}$, define three classical centered \linebreak Bernoulli  random walks. Considered together, they form a noncommutative Markov process which converges, after a proper renormalization, towards a standard Brownian motion in $\R^3$ (see  Biane \cite{bianemarchesbernoulliq} for more details).  Furthermore, the  family of noncommutative random variables
\begin{equation}\label{d=2noncom} (Z_n,\sqrt{X_n^2+Y_n^2+Z_n^2}, \,n\ge 1),\end{equation}  
forms a discrete-time approximation of the Markov process (\ref{d=2}). 
Since the noncommutative process (\ref{d=2noncom}) is also Markovian (see \cite{bianemeyer}), there is a quite noticing analogy between what happens in the commutative and noncommutative cases.

In higher dimension, there are several natural ways to generalize the construction of processes (\ref{d=2}) and  (\ref{d=2noncom}). For some of them, the Markov property fails. For instance for $d\ge 2$, in the commutative case, if $(M(t),t\ge 0)$ is a $d\times d$ Hermitian Brownian motion, the process obtained by considering the eigenvalues of two consecutive minors of $(M(t)$, $t\ge 0)$, is Markovian whereas the Markovianity fails if one considers more than two consecutive minors, as it has been recently proved  in \cite{Nordenstam}, and   announced  in \cite{manon}. This result has  also an analogue in a noncommutative framework as we shall see in the sequel.

 In this paper we extend to higher dimensions the construction of the noncommutative process (\ref{d=2noncom}).
 For this we need some basic facts about representation theory of Lie algebra recalled in section \ref{Universal enveloping algebra}.   In section \ref{Quantum Markov chain} we recall the construction of quantum Markov chains. The Markovian aspects are studied more specifically in section \ref{Restriction to a subalgebra} using some existing results of invariant theory. In particular we discuss the Markovianity  of noncommutative analogues of the processes of eigenvalues of consecutive minors. In the last section, considering the limit of the noncommutative processes previously studied, we discuss the  Markovianity of some natural generalizations of  the process (\ref{d=2}). 
 
\section{Universal enveloping algebra}\label{Universal enveloping algebra}

Let $G=\operatorname{GL}_d(\mathbb C)$ be the group of $d\times d$ invertible matrices, and $\mathfrak{g}=\operatorname{M}_d(\mathbb C)$ its Lie algebra, which is the algebra of $d\times d$ complex matrices.
 Letting $e_{ij}$, $i,j=1,\dots,d$, be the standard basis in $\operatorname{M}_d(\mathbb{C})$, the universal enveloping algebra  $\mathcal U(\mathfrak{g})$ of $\mathfrak{g}$ is the associative algebra generated by $e_{ij}$, $i,j=1,\dots,d$, with 
 no relations among the generators other than the following 
  commutation relations
\[
[e_{ij},e_{kl}]=\delta_{jk}e_{il}-\delta_{il}e_{kj},
\]
where $[\cdot,\cdot]$ is the usual bracket   of $\mathfrak g$. By the Poincar\'e-Birkhoff-Witt theorem (see \cite{zelo}),   there exists a basis of $\mathcal U(\mathfrak g)$ composed of monomials 
\[
e_{i_1j_1}\cdots e_{i_mj_m},
\]
where the integers $i_k,j_k$ are taken in a certain order. Hence, writing an element of $\mathcal U(\mathfrak g)$ in this basis, its degree is defined as the degree of its leading term.
For $n\in \N$, we denote $\mathcal U_n(\mathfrak{g})$ the set of elements  of $\mathcal U(\mathfrak{g})$ whose leading term is of degree smaller than $n$.  Recall that a representation of $\mathfrak g$ in a finite dimensional vector space $V$ is a Lie algebra homomorphism
\[
\rho\colon \mathfrak g\to \operatorname{End}(V).
\]
Then any representation $\rho$ of $\mathfrak{g}$ extends uniquely to the universal enveloping algebra letting 
\[
\rho(xy)=\rho(x)\rho(y),\quad x,y\in \mathcal U(\mathfrak{g}).
\]
Let $I$ be the  identity  matrix of size $d\times d$. 
The coproduct on $\mathcal U(\mathfrak g)$ is the algebra homomorphism $\Delta\colon \mathcal U(\mathfrak g) \to \mathcal U(\mathfrak g)\otimes \mathcal U(\mathfrak g)$ defined on the generators by
\begin{align*}
\Delta (I)&= I\otimes I\\
\Delta (e_{ij})&= I\otimes e_{ij}+ e_{ij}\otimes I,\, \textrm{
if $i\ne j$, }i,j=1,\dots,d\\
\Delta (h_{i})&= I\otimes h_{i}+ h_{i}\otimes I, \, i=1,\dots,d-1,
\end{align*}
 where $h_i=e_{ii}-e_{i+1 i+1}$. This characterizes entirely  $\Delta$ 
 letting 
\[
\Delta(xy)=\Delta(x)\Delta(y), \quad x,y\in \mathcal{U}(\mathfrak{g}),
\]
where the product on  $\mathcal U(\mathfrak g)\otimes \mathcal U(\mathfrak g)$ is defined in the usual way $(a\otimes b) (c\otimes d)=ac\otimes bd$, for $a,b,c,d\in\mathcal U(\mathfrak g)$.
The tensor product of two representations  $\rho_1\colon \mathfrak g\to \operatorname{End}(V_1)$ and $\rho_2\colon\mathfrak g\to\operatorname{End}(V_2)$ and its extension to  $\mathcal U(\mathfrak g)$
\[
\rho_1\otimes\rho_2\colon    \mathcal{U}(\mathfrak g) \to \operatorname{End}(V_1\otimes V_2)
\]
 is given by
\[
\rho_1\otimes \rho_2(x)=(\rho_1\otimes \rho_2)\Delta(x), \quad x\in \mathcal U(\mathfrak g),
\]
where $(\rho_1\otimes \rho_2)(x_1\otimes x_2)=\rho_1(x_1)\otimes \rho_2(x_2)$, for $x_1,x_2\in    \mathcal U(\mathfrak g)$. 
For a representation $\rho$ of $\mathfrak{g}$, we define recursively the representation $\rho^{\otimes n}$ of $\mathcal{U}(\mathfrak{g})$ by
\[
\rho^{\otimes n}(x):=(\rho^{\otimes n-1}\otimes\rho)\Delta(x), \quad x\in \mathcal{U}(\mathfrak{g}).
\]

\section{Quantum Markov chain}\label{Quantum Markov chain}

We first recall some basic facts about noncommutative probability, which can be found in \cite{meyerquantum} for example. 
A noncommutative probability space $(\mathcal A,\varphi)$ is composed of a unital $^\ast$-algebra, and a state $\varphi\colon \mathcal A\to\C$, that is a positive linear form,  in the sense that $\varphi(aa^*)\geq0$ for all $a\in\mathcal A$, and normalized, \ie $\varphi(1)=1$.  Elements of $\mathcal A$ are called noncommutative random variables. 
Note that classical probability is recovered, at least for bounded random variables,  by letting $\mathcal A=L^\infty(\Omega,\P)$ for some probability space $(\Omega,\P)$, and $\varphi$ being the expectation $\E$. The law of a family $(a_1,\ldots,a_n)$ of noncommutative random variables is defined as the collection of $^\ast$-moments
\[
\varphi(a_{i_1}^{\varepsilon_1}\cdots a_{i_k}^{\varepsilon_k}),
\]
where for all $j=1,\ldots,k$, $i_j\in\{1,\ldots,n\}$, $\varepsilon_j\in\{1,\ast\}$,  and $k\geq1$. Thus,
convergence in distribution means convergence of all $^\ast$-moments.

Recall that a von Neumann algebra is a subalgebra of the algebra of bounded operators on some Hilbert space, closed under the strong topology.
Define $\mathcal W=\operatorname{M}_d(\C)^{\otimes \infty}$ the infinite tensor product  in the sense of von Neumann algebras, with respect to the product state $\omega=\tr^{\otimes \infty}$, where $\tr=\frac{1}{d}\Tr$ is the normalized trace on $\operatorname{M}_d(\C)$. Hence, $(\mathcal W,\omega)$ is  a noncommutative probability space. For $a_1,\dots,a_n\in \operatorname{M}_d(\C)$, we use the notation $a_1\otimes\dots\otimes a_n$ instead of $a_1\otimes\dots\otimes a_n\otimes I^{\otimes \infty}.$ Let us now recall the construction of quantum Markov chains, as it can be found in \cite{bianemeyer}. First, let us see how classical Markov chains can be translated in the noncommutative formalism. If $(X_n)_{n\geq1}$ is a  classical Markov chain defined on some probability space $(\Omega,\P)$ and taking values in a measurable space $E$, then for each $n\geq1$, the random variable $X_n\colon \Omega\to E$ gives rise to an algebra homomorphism 
\begin{align*}
\chi_n\colon L^\infty(E)&\to L^\infty(\Omega)\\
f\ \ &\mapsto f(X_n).
\end{align*}
 Hence, one can think of a noncommutative random variable as an algebra homomorphism. The Markov property of $(X_n)_{n\geq1}$ writes
 \[
 \E(Yf(X_{n+1}))=\E(Y Qf(X_n)),
 \]
for all $\sigma(X_1, \ldots, X_n)$-measurable random variable $Y$, and where $Q\colon L^\infty(\Omega)\to L^\infty(\Omega)$ is the transition operator of $(X_n)_{n\geq1}$. Translating this in the homomorphism formalism, we get
\[
\E(\psi \chi_{n+1} (f))=\E(\psi \chi_{n}(Qf)),
\] 
where $\psi$ is in the subalgebra of $L^\infty(\Omega)$ generated by $X_1,\ldots,X_n$. 

Let us pass to the construction properly speaking of the quantum Markov chain considered here.  Let $\rho$ be the standard representation of $\mathfrak{g}$. We consider the morphism
\begin{align*}
j_n\colon \mathcal U(\mathfrak{g})&\to \mathcal W\\
x &\mapsto  \rho^{\otimes n}(x),
\end{align*}
for all $n\geq1$. Define
$P\colon \mathcal{U}(\mathfrak{g})\to \mathcal U(\mathfrak{g})$ by
\[
P=\id \otimes\eta \circ \Delta,
\]
 where $\eta(\cdot)=\tr(\rho(\cdot))$. $P$ is a unital completely positive map, which is the analogue of Markov operator in the quantum context. We have that  $(j_n)_{n\geq1}$ is a quantum Markov chain, in the sense that it satisfies the following Markov property.
\begin{prop}
For all $\xi$ in the von Neumann algebra generated by $\{j_k(\mathcal U(\mathfrak{g})), k\leq n-1\}$, and all $x\in\mathcal U(\mathfrak{g})$,
\[
\omega(j_n(x)\xi)=\omega(j_{n-1}(Px)\xi).
\]
\end{prop}
\begin{proof}
Let  $\xi=a_1\otimes\cdots\otimes a_{n-1}$, where the $a_i$'s are in $\operatorname{M}_d(\C)$. Using Sweedler's notation
\[
\Delta (x) =\sum x^1\otimes x^2,
\] we have on one hand
\begin{align*}
\omega(j_n(x)\xi)&=\omega((\rho^{\otimes n-1}\otimes\rho)\Delta(x)\xi)\\
&=\sum \omega\left(\rho^{\otimes n-1}(x^{1})\otimes \rho(x^{2}) \xi\right),
\end{align*}
so,
 \begin{align*}
\omega(j_n(x)\xi)=\sum \tr(\rho^{\otimes n-1}(x^{1}) \xi)\tr(\rho(x^{2})).
\end{align*} 
On the other hand, \[ Px=\sum x^1\eta(x^2). \]
Thus 
\[
j_{n-1}(Px)=\sum \eta(x^2) \rho^{\otimes n-1}(x^{1}),
\]
and 
 
\[
\omega (j_{n-1}(Px)\xi)=\sum \tr\left(\rho^{\otimes n-1}(x^{1})\xi\right)\tr\left(\rho(x^2)\right),
\]
which achieves the proof.
\end{proof}

\section{Restriction to a subalgebra}\label{Restriction to a subalgebra}

Recall that the group $G$ acts on $\mathfrak{g}$ via the adjoint action, \ie the conjugation action, given by
\[
\operatorname{Ad}(g)x=gxg^{-1},\quad g\in G,\, x\in \mathfrak{g}.
\]
 This action extends uniquely to an action on $\mathcal U(\mathfrak{g})$ letting  \[\Ad(g)(xy)=(\Ad(g)x)(\Ad(g)y),\quad g\in G,\, x\in \mathcal{U}(\mathfrak{g}).\]
The group $G$ acts on $ \mathcal{U}(\mathfrak{g})\otimes  \mathcal{U}(\mathfrak{g})$ via the action
\[ \Ad(g)(x\otimes y)=(\Ad(g)x)\otimes (\Ad(g)y),\quad g\in G,\, x,y\in \mathcal{U}(\mathfrak{g}).\]
Note that the morphism $\Delta$ satisfies 
\begin{align}
\Delta(\Ad(g)x)=\Ad(g)\Delta(x),\quad g\in G, \,x\in\mathcal{U}(\mathfrak{g}).
\end{align}
The next proposition shows that the operator $P$ commute with  the adjoint action.

\begin{prop} \label{Commutation} For all $g\in G$, and all $x\in\mathcal{U}(\mathfrak{g})$, we have
\[
\Ad(g)P(x)=P(\Ad(g)x).
\]
\end{prop}
\begin{proof} Using the notation $\Delta x=\sum x^1\otimes x^2$ for $x\in\mathcal U(\mathfrak{g})$, we have 
\[
\Ad(g)P(x)=\Ad(g)\left(\sum x^1\eta(x^2)\right)=\sum \Ad(g)x^1\eta(x^2),
\]
and
\begin{align*}
P(\Ad(g)x)&=\id\otimes\eta\circ \Delta(\Ad(g)x)=\id\otimes\eta\left(\Ad(g)\Delta x\right)\\
&= \sum \Ad(g)x^1\eta (x^2), 
\end{align*}
since $\eta$ is a trace.
\end{proof}

 \begin{defi} For a subgroup $K$ of $G$, an element $x\in \mathcal{U}(\mathfrak{g})$ is said to be $K$-invariant if \[\Ad(g)x=x, \, \forall g\in K.\] 
 \end{defi}
 The set of $K$-invariant elements of  $\mathcal{U}(\mathfrak{g})$ is denoted $\mathcal{U}(\mathfrak{g})^K$. For $n\in \N$, we denote $\mathcal U_n(\mathfrak{g})^K$ the subset of $\mathcal U(\mathfrak{g})^K$ of elements whose leading term is of degree smaller than $n$. Proposition \ref{Commutation} implies the following one, which is fundamental for our purpose.
 \begin{prop}\label{UKstable}
 Let $K$ be a subgroup of $G$. The subalgebra $\mathcal U(\mathfrak g)^K$ of $K$-invariant elements of $\mathcal U(\mathfrak g)$ is stable by $P$, \ie
 \[
 P\, \mathcal U(\mathfrak g)^K\subset \mathcal U(\mathfrak g)^K.
 \]
 Hence, the restriction of $(j_n)_{n\geq1}$ to $\mathcal U(\mathfrak g)^K$ defines a quantum Markov chain.
 \end{prop}

Let us focus on  some particular invariant sets related to the minor process studied in \cite{Nordenstam}. For a fixed integer $p\in \{0,\dots,d-1\}$,  we consider the block diagonal subgroup $\operatorname{GL}_{d-p}(\mathbb{C})\times \mathbb{C^*}^{p}$ of $G$ which consists of elements of the form
\[
\left(
\begin{array}{cccc}
k & & \phantom{\ddots} 
\multirow{2}{*}{\huge 0}&  \\
& k_1 & &   \\
 & \multirow{2}{*}{\huge 0}\ \ & \ddots & \\
  &  & &  k_p
\end{array}
\right),
\]
with $k\in \operatorname{GL}_{d-p}(\mathbb{C})$, and $k_1,\ldots, k_p\in\C^*$.
For $l,m\in \N^*$, we denote $\mathcal{M}_{l,m}$ the set of $l\times m$ matrices with noncommutative entries in $\mathcal{U}(\mathfrak{g})$. We let $\mathcal{M}_{l}=\mathcal{M}_{l,l}$. The rules to add or multiply matrices of $\mathcal{M}_{l,m}$ are the same as those for the commutative case replacing the usual addition and multiplication in a commutative algebra by the addition and the multiplication in $\mathcal{U}(\mathfrak{g})$. Moreover, if $M=(m_{ij})_{1\le i,j\le l}$ is a matrix in $\mathcal{M}_{l}$,  then the element of $\mathcal{U}(\mathfrak{g})$ equal to $\sum_{i=1}^l m_{ii}$ is denoted ${\bf Tr}(M)$. We partition the matrix  $E=(e_{ij})_{1\le i,j\le d}$  in block matrices in the form
 \[
 E=  \left[\begin{array}{ccc}
    E_{11} & \dots & E_{1p+1} \\ 
    \vdots &  & \vdots \\ 
    E_{p+11} & \dots & E_{p+1p+1} \\ 
  \end{array}\right],
  \]
 where  $E_{11}\in \mathcal{M}_{d-p,d-p}$, $E_{1i}\in \mathcal{M}_{1,d-p}$, $E_{i1}\in \mathcal{M}_{d-p,1}$, $i\in\{2,\dots, p+1\}$, and 
 $E_{ij}\in \mathcal{U}(\mathfrak{g})$, $i,j\in\{ 2,\dots,d\}$.

 \begin{notaetoile}
 Entries of a matrix will be always  denoted by small letters, while capital letters will refer to the partition defined above.
 \end{notaetoile}
 
 The next theorem, which has been proved by Klink and Ton-That, gives the generators of the subalgebra  $\mathcal{U}(\mathfrak{g})^{\operatorname{GL}_{d-p}(\mathbb{C})\times \mathbb{C^*}^{p}}$.
 
\begin{thm}[\cite{Klink}] \label{theoklink}The subalgebra  $\mathcal{U}(\mathfrak{g})^{\operatorname{GL}_{d-p}(\mathbb{C})\times \mathbb{C^*}^{p}}$ is finitely generated by the constants  and elements 
\[
{\bf Tr}(E_{i_1i_2}\cdots E_{i_qi_1}),\quad q\in \N^*, \, i_1,\dots,i_q\in\{1,\dots,p+1\}.
\]
\end{thm}
The two extreme cases  of the above theorem  give the following classical results. Actually for $p=0$,  it implies that the center of  $\mathcal{U}(\mathfrak{g})$ is generated by  Casimir operators (see \cite{zelo})
\[
{\bf Tr}(E^k),\quad k\in \N.
\]  
For $p=d-1$, we recover that the commutant of  $\{e_{ii},\, i=1,\dots,d\}$ in $\mathcal{U}(\mathfrak{g})$ is generated by  elements 
\[
e_{i_1i_2}\cdots e_{i_qi_1},\quad q\in \N, \, i_1,\dots,i_q\in\{1,\ldots,p+1\}.
\]
  For $a_1,\dots,a_n\in \mathcal{U}(\mathfrak{g})$, we denote 
  \[
  \langle a_1,\dots,a_n\rangle,\] 
  the subalgebra of $\mathcal{U}(\mathfrak{g})$ generated by the constants and elements $a_1,\dots,a_n$. Let us focus on the subalgebra $\mathcal{U}(\mathfrak{g})^{\text{GL}_{d-p}(\mathbb{C})\times \mathbb{C^*}^{p}}$ and its generators in the case when $p=1$ and $p=2$.
First we need the following lemmas.  
\begin{lem}\label{CommuteTrace} 
If $A=(a_{ij})_{1\leq i,j\leq d},B=(b_{ij})_{1\leq i,j\leq d}\in\mathcal{M}_d$, with $a_{ij}\in \mathcal{U}_n(\mathfrak{g})$, $b_{ij}\in\mathcal{U}_m(\mathfrak{g})$, for $i,j=1,\dots,d$,  then
\[
{\bf Tr}( AB)-{\bf Tr}(BA)\in \mathcal{U}_{m+n-1}(\mathfrak{g}).
\]
\end{lem}
\begin{proof}
This is a consequence of the commutation relations in $\mathcal U(\mathfrak g)$. 
\end{proof}
The following lemma claims that the subset of invariants  $\mathcal{U}(\mathfrak{g})^{\text{GL}_{d-1}(\mathbb{C})\times \mathbb{C^*}}$ is generated by the Casimir elements associated to the Lie algebra $\operatorname{M}_{d}(\mathbb{C})$ and those associated to the subalgebra $\{M\in\operatorname{M}_d(\mathbb{C}): m_{id}=m_{di}=0, i=1,\dots,d\}\simeq \operatorname{M}_{d-1}(\mathbb{C})$. 
 \begin{lem}\label{ZdZd-1}
The subalgebra $\mathcal{U}(\mathfrak{g})^{\operatorname{GL}_{d-1}(\mathbb{C})\times \mathbb{C^*}}$ is generated by 
\[
{\bf Tr}(E_{11}^{k-1}),\, {\bf Tr} (E^k), \quad k=1,\dots,d.
\]
\end{lem}
\begin{proof}
For $q\in \N^*$, let $\mathcal{T}_q$ be the subalgebra 
\[
\langle {\bf Tr}(E_{i_1i_2}\cdots E_{i_ki_1}),\, k\in\{1,\dots q\},\, i_1,\dots,i_k=1,2\rangle.
\]
 It is sufficient to prove that for every $q\in \N^*$ \begin{align}\label{Hq} \mathcal{T}_q=\langle {\bf Tr}(E_{11}^{k}), {\bf Tr}(E^{k}),\,k\in\{1,\dots, q\}\rangle.\end{align}  For every $q\in \N^*$ the inclusion 
 \[
 \langle {\bf Tr}(E_{11}^{k}), {\bf Tr}(E^{k}),\,k\in\{1,\dots, q\}\rangle\subset \mathcal{T}_q
 \]
   follows from the fact that 
   \[{\bf Tr}(E^k)=\sum {\bf Tr}(E_{i_1i_2}\cdots E_{i_ki_1}),\]
  where the sum runs over all sequences $i_1,\dots,i_k$ of integers in $\{1,2\}$. Let us prove the reverse inclusion by induction  on $q$. It is clearly true for $q=1$. For $q=2$, let us write 
\[
{\bf Tr}(E^2)={\bf Tr}(E_{21}E_{12})+{\bf Tr}(E_{12}E_{21})+{\bf Tr}(E_{11}^2)+{\bf Tr}(E_{22}^2).
\]
Thus the inclusion 
\[
\mathcal{T}_2 \subset\langle {\bf Tr}(E_{11}^{k}), {\bf Tr}(E^k),\,k=1,2\rangle
\]
 follows from Lemma \ref{CommuteTrace} which implies that 
 \[
 {\bf Tr}(E_{21}E_{12})-{\bf Tr}(E_{12}E_{21})\in \mathcal{U}^{\text{GL}_{d-1}\times \mathbb{C^*}}_1(\mathfrak{g})\subset\langle1,E_{11},E_{22}\rangle.
 \]
The case $q=3$ is proved in a similar way. 
Suppose that (\ref{Hq}) is true for $q-1$, for a fixed $q\ge 4$. Let $i_1,i_2\dots,i_q,$ be a sequence of integers in $\{1,2\}$.
If  the sequence $i_1,i_2\dots,i_q,$ contains no successive integers equal to $1$ then $E_{i_1i_2}\cdots E_{i_qi_1},$ contains only factors equal to  $E_{21}E_{12}$, $E_{12}E_{21}$, or $E_{22}$. By lemma \ref{CommuteTrace} and inclusion
\begin{align}\label{inclusion}
\mathcal{U}_{q-1}(\mathfrak{g})^{\text{GL}_{d-1}(\C)\times \mathbb{C^*}}\subset \mathcal{T}_{q-1},
\end{align}
we can suppose that $E_{i_1i_2}\cdots E_{i_qi_1},$ contains only factors equal to  $E_{21}E_{12}$, or $E_{22}$, which belongs to the subalgebra 
\[
\langle {\bf Tr}(E_{11}^{k}), {\bf Tr}(E^k), k=1, 2\rangle.
\]
If $E_{i_1i_2}\cdots E_{i_qi_1},$ contains factors equal to $E_{11}$ but strictly less than $q-2$, then it contains at least one factor equal to $E_{21}E_{11}E_{12}$, $E_{11}E_{12}E_{21}$ or $E_{12}E_{21}E_{11}$. Thanks to lemma \ref{CommuteTrace}, and inclusion (\ref{inclusion}) we can suppose that $i_1=i_3=2$ and $i_2=1$. Thus
\[
{\bf Tr }(E_{i_1i_2}\cdots E_{i_qi_1})=(E_{21}E_{11}E_{12}){\bf Tr}(E_{2i_4}\cdots E_{i_q2}).
\]
Then the induction hypothesis implies 
\[
{\bf Tr }(E_{i_1i_2}\cdots E_{i_qi_1})\in \langle  {\bf Tr}(E_{11}^{k}), {\bf Tr}(E^k),k\in\{1,\dots,q-1\}\rangle.
\]
If $E_{i_1i_2}\cdots E_{i_qi_1}=E^d_{11}$ then  
\[
{\bf Tr }(E_{i_1i_2}\cdots E_{i_qi_1})={\bf Tr }(E_{11}^{q})\in \langle  {\bf Tr}(E_{11}^{k}), {\bf Tr}(E^k),k\in\{1,\dots, q\}\rangle.
\]
If $E_{i_1i_2}\cdots E_{i_qi_1}$ contains $q-2$ factors equal to $E_{11}$, then
\[
{\bf Tr }(E_{i_1i_2}\cdots E_{i_qi_1})\in \{{\bf Tr}(E_{11}^{q-2}E_{12}E_{21}),{\bf Tr}(E_{21}E_{11}^{q-2}E_{12}),{\bf Tr}(E_{12}E_{21}E_{11}^{q-2})\}.
\]
We write
\begin{align*} 
{\bf Tr}(E^q)={\bf Tr}(E_{11}^{q})+{\bf Tr}(E_{11}^{q-2}&E_{12}E_{21})+{\bf Tr}(E_{21}E_{11}^{q-2}E_{12})+{\bf Tr}(E_{12}E_{21}E_{11}^{q-2})\\
&+\sum{\bf Tr }(E_{i_1i_2}\cdots E_{i_qi_1})
\end{align*}
where the sum runs over all sequences $i_1,\dots,i_q$ of integers in $\{1,2\}$ containing strictly less than $q-1$ integers equal to $1$. The previous cases, Lemma \ref{CommuteTrace} and inclusion (\ref{inclusion}) imply that
\[
{\bf Tr}(E^q)-3{\bf Tr }(E_{i_1i_2}\cdots E_{i_qi_1})\in \langle  {\bf Tr}(E_{11}^{k}), {\bf Tr}(E^k),k\in\{1,\dots, q\}\rangle.
\]
Since it is known (see \cite{zelo}) that
\[
\langle{\bf Tr}(E^k),k\in\{1,\dots, d\}\rangle=
\langle{\bf Tr}(E^k),k\ge 1\rangle,
\] and 
\[
\langle{\bf Tr}(E_{11}^k),k\in\{1,\dots, d-1\}\rangle=
\langle{\bf Tr}(E_{11}^k),k\ge 1\rangle,
\]
the proposition follows.
\end{proof}
\begin{rem} \label{p>2}
When $p\ge 2$, the subalgebra  $\mathcal{U}(\mathfrak{g})^{\text{GL}_{d-p}(\mathbb{C})\times  \mathbb{C^*}^{p}}$ is not generated by the Casimir elements associated to the Lie algebras  $\{M\in\operatorname{M}_d(\mathbb{C}): m_{ij}=m_{ji}=0, i=1,\dots,d,j=d-k+1,\dots,d\}\simeq \operatorname{M}_{d-k}(\mathbb{C})$, $k\in\{1,\dots,p\}$. For instance, 
\[
{\bf Tr}(E_{13}E_{31}) \notin \langle {\bf Tr}(E_{11}^k),{\bf Tr} (
\left[
\begin{array}{cc}
E_{11}  &   E_{12}   \\
 E_{21} & E_{22}        
\end{array}
\right]^k), {\bf Tr}(E^k),\,k\in \N\rangle
\] and thus
\[
\langle {\bf Tr}(E_{11}^k),{\bf Tr} (
\left[
\begin{array}{cc}
E_{11}  &   E_{12}   \\
 E_{21} & E_{22}        
\end{array}
\right]^k), {\bf Tr}(E^k),\,k\in \N\rangle\subsetneq \mathcal{U}(\mathfrak{g})^{\operatorname{GL}_{d-2}(\C)\times\mathbb{C^*}^2}.
\]
\end{rem} 
The following theorem is a quantum analogue of    theorems 2.2  of \cite{Nordenstam}.
\begin{thm} \label{Restriction for p=1} The restriction of the $j_n$'s to the subalgebra \[\langle {\bf Tr}(E_{11}^{k-1}), {\bf Tr}(E^{k}),\,k\in\{1,\dots, d\}\rangle,\] defines a quantum Markov process.
\end{thm} 
\begin{proof} Theorem follows immediately from  proposition \ref{UKstable} and lemma \ref{ZdZd-1}.\end{proof}
Note that the subalgebra 
\[
\langle {\bf Tr}(E_{11}^{k-1}), {\bf Tr}(E^{k}),\,k\in\{1,\dots, d\}\rangle,
\] 
is commutative. Thus, as in \cite{bianemeyer} which focus on the $d=2$ case, the quantum Markov process in the above theorem is a noncommutative process, with a commutative Markovian operator. Taking $d=2$ in   theorem \ref{Restriction for p=1}   the Markovianity of the process (\ref{d=2noncom}) follows.  The following theorem is an analogue of    theorem 2.4  of \cite{Nordenstam} in a noncommutative context. The non-Markovianity comes from  remark \ref{p>2}.
\begin{thm}  \label{Restriction for p=2} The restriction of the $j_n$'s to the subalgebra \[\langle {\bf Tr}(E_{11}^k),{\bf Tr} (
\left[
\begin{array}{cc}
E_{11}  &   E_{12}   \\
 E_{21} & E_{22}        
\end{array}
\right]^k), {\bf Tr}(E^k),\,k\in\N\rangle,\] does not define a quantum Markov process.
\end{thm} 
\begin{proof}
We have to prove that the subalgebra 
\[
\mathcal B:=
\langle {\bf Tr}(E_{11}^k),{\bf Tr} (
\left[
\begin{array}{cc}
E_{11}  &   E_{12}   \\
 E_{21} & E_{22}        
\end{array}
\right]^k), {\bf Tr}(E^k),\,k\in\N\rangle,
\] is not stable by the operator $P$. Indeed,  the partition of $E$ for $p=2$ writes
 \[
 E=\left[
 \begin{array}{ccc}
 E_{11} &  E_{12} &  E_{13}\\
  E_{21} &  E_{22} &  E_{23}\\
   E_{31} &  E_{32} &  E_{33}
 \end{array}
 \right].
 \]
One can prove by straightforward calculation that the element
 \[
 a=
 E_{21}E_{12}\left(E_{31}E_{13}+E_{32}E_{23}\right)^2
 \] 
 is in $\mathcal B$, but $Pa$ does not, which proves the theorem.
\end{proof}
Let us choose an integer $m$ large enough such that the subalgebras
\[
\langle {\bf Tr}(E_{i_1i_2}\cdots E_{i_qi_1}),\quad q\in \N^*, \, i_1,\dots,i_q\in\{1,\dots,p+1\}\rangle
\] 
and 
\[
\langle {\bf Tr}(E_{i_1i_2}\cdots E_{i_qi_1}),\quad q=1,\dots,m, \, i_1,\dots,i_q\in\{1,\dots,p+1\}\rangle
\]
are equal. 
In the framework of this paper, the natural process which "contains" the one of theorem \ref{Restriction for p=2} and remains Markovian, is given in the following theorem taking $p=2$. 
\begin{thm}The restriction of the $j_n$'s to the subalgebra 
\[
\langle{\bf Tr}(E_{i_1i_2}\cdots E_{i_qi_1}),\quad q\in\{1,\dots,m\}, \, i_1,\dots,i_q\in\{1,\dots,p+1\}\rangle,
\] 
defines a quantum Markov process.
\end{thm} 
\begin{proof}Theorem follows from theorem \ref{theoklink} and proposition \ref{UKstable}. \end{proof}
 \section{Random matrices}\label{Applications to random matrices}
 Let  $\operatorname{H}_d$ and $\operatorname{H}^0_d$ be respectively the set of $d\times d$ complex  Hermitian matrices and the set of $d\times d$ complex  Hermitian matrices with null trace, both endowed with the scalar product given by \[\langle M,N\rangle= \Tr(MN), \quad M,N\in  \textrm{ $\operatorname{H}_d$ (resp. $\operatorname{H}^0_d)$.}\] 

 For $k,l\in \N^*$, we denote $\operatorname{M}_{k,l}(\C)$ the set of $k\times l$ complex matrices and let $\operatorname{M}_{k}(\C)=\operatorname{M}_{k,k}(\C)$. As in the noncommutative case, we partition a matrix  $M\in \operatorname{M}_d(\C)$ in block matrices in the form
 \[
 M=  \left[\begin{array}{ccc}
    M_{11} & \dots & M_{1p+1} \\ 
    \vdots &  & \vdots \\ 
    M_{p+11} & \dots & M_{p+1p+1} \\ 
  \end{array}\right],
  \]
 where  $M_{11}\in \operatorname{M}_{d-p,d-p}(\C)$, $M_{1i}\in \operatorname{M}_{1,d-p}(\C)$, $M_{i1}\in \operatorname{M}_{d-p,1}(\C)$, $i\in\{2,\dots, p+1\}$, and 
 $M_{ij}\in \C$, $i,j\in\{ 2,\dots,d\}$.

Define the elements $(x_{ij})_{1\leq i,j\leq d}$ of $\mathcal U(\mathfrak g)$ by
\[
x_{ij}=e_{ij}, \ \text{ for $1\leq i\not= j\leq d$,} \ \text{ and }
x_{ii}=e_{ii}-\frac{1}{d}\, I, \ \text{ for $1\leq i\leq d$}.
\]
Note all the $x_{ij}$'s are traceless elements of $\mathfrak g$. Let $v=\frac{d}{d-1}\tr(\rho(x_{ii})\rho(x_{ii}))$ which does not depend on $i$. Then we have the following theorem which is due to Biane.
  \begin{thm}[Biane, \cite{bianepermutation}] The law of the family of random variables on $(\mathcal{W},\omega)$
  \[
  \left(\frac{1}{\sqrt{nv}}j_{\lfloor nt\rfloor}(x_{ij})\right)_{t\in \R_+,1\le i,j\le d}
  \] 
  converges as $n$ goes to infinity towards the law of 
  \[
  (m_{jk}(t))_{t\in \R_+,1\le i,j\le d},
  \]
  where $(M(t)=(m_{ij}(t))_{1\leq i,j\leq d},t\ge 0)$ is a standard  Brownian motion on $\operatorname{H}_d^0$.
 \end{thm}

By the above theorem, we see that the law of the noncommutative process 
\begin{align}\label{noncomproc}
\Big(\frac{1}{\sqrt{n v}}j_{\lfloor nt\rfloor}\Big)_{t\geq0}
\end{align}
 restricted to the subalgebra of theorem \ref{Restriction for p=1} converges, as $n$  goes to infinity, towards the law of $(\Tr(M_{11}(t)^{k-1}), \Tr(M(t)^k), k\geq1, t\geq0)$. We will see that this process, which is equivalent to the process of  eigenvalues of two consecutive minors of $(M(t),t\geq0)$, is Markovian. More generally, if $K$ is a subgroup of $G$, the law of the noncommutative process (\ref{noncomproc}) restricted to the subalgebra  $\mathcal{U}(\mathfrak{g})^{K}$ converges, as $n$  goes to infinity, to a commutative process which remains Markovian. The fact that the limit process is a Markov process will follow by
     It\^o's calculus and invariant theory in a commutative framework. A function $f\colon\text{M}_d(\C)\to \C$ is seen as a function from $\C^{d^2}$ to $\C$. 
   
    \begin{defi} Let  $K$ be a subgroup of $G$. A function $f$ from $\operatorname{M}_d(\mathbb{C})$ to $\C$ is said to be $K$-invariant if  
  \begin{align*}
 \forall \,k\in K  \quad \forall M\in \operatorname{M}_d(\mathbb{C}), \quad f(kMk^{-1})=f(M).
 \end{align*} 
 \end{defi}

Let $\mathcal{P}(\mathfrak{g})$ denote the algebra of all complex-valued polynomial functions on $\mbox{M}_d(\mathbb C)$, \ie $\mathcal{P}(\mathfrak{g})$  is the set  of all polynomials in coordinates of a matrix of $\text{M}_d(\C)$.  For any subgroup $K$ of $G$, the set of $K$-invariant elements of  $\mathcal{P}(\mathfrak{g})$ is denoted $\mathcal{P}(\mathfrak{g})^K$.  
The following theorem, which is a commutative version of theorem \ref{theoklink}, has been proved in (\cite{Klink}).
  
\begin{thm}[\cite{Klink}] \label{invcom} It exists $m\in \N$, such that the subalgebra  $\mathcal{P}(\mathfrak{g})^{\operatorname{GL}_{d-p}(\mathbb{C})\times \mathbb{C^*}^{p}}$ is   generated by the constants and polynomials 
\[
M\in \operatorname{M}_d(\C)\mapsto\Tr(M_{i_1i_2}\cdots M_{i_qi_1}),\quad q\in\{1\dots,m\}, \, i_1,\dots,i_q\in\{1,\dots,p+1\}.
\] 
\end{thm}

 Let us recall the following property of Brownian motion and invariant functions. In what follows we denote by $\langle \cdot, \cdot\rangle$ the usual quadratic covariation, and by $\operatorname d$ and $\operatorname d^2$ the usual first and second order differentials. 
 \begin{prop} 
 Let $g\in \operatorname{GL}_d(\C)$, and $f$ and $h$ be twice differentiable functions from $\operatorname{M}_d(\C)$ to $\C$ such that 
 \begin{align} \label{invg} \forall M\in \operatorname{M}_d(\C) \quad f(gMg^{-1})=f(M) \textrm{ and  }\,  h(gMg^{-1})=h(M).\end{align} If $B$ is a standard Brownian motion on $\operatorname{H}_d$, then
\[
 \langle \operatorname d\!f(gMg^{-1})(dB),  \operatorname d\!h(gMg^{-1})(dB)\rangle = \langle \operatorname d\!f(M)(dB),  \operatorname d\!h(M)(dB)\rangle
 \]
 and
 \[
 \langle \operatorname d^2\!f(gMg^{-1})(dB),dB\rangle = \langle \operatorname d^2\!f(M)(dB),dB\rangle.
 \]
 \end{prop}
 \begin{proof}   Since $B$ is a standard Brownian motion on $\operatorname{H}_d$, 
 \[
 \langle (gdBg^{-1})_{ij},(gdBg^{-1})_{kl}\ra= \la dB_{ij},dB_{kl}\ra, \quad i,j,k,l\in \{1,\dots,d\}.
 \]
 Thus 
 \begin{align*}
 \langle \operatorname d\!f(gMg^{-1})(dB),\operatorname d\!h(gMg^{-1})(dB)\rangle&= \langle \operatorname d\!f(gMg^{-1})(gdBg^{-1}), \operatorname d\!h(gMg^{-1})(gdBg^{-1})\rangle,
 \end{align*}
 and
  \begin{align*}
 \langle \operatorname d^2\!f(gMg^{-1})(dB),dB\rangle&= \langle \operatorname d^2\!f(gMg^{-1})(gdBg^{-1}),gdBg^{-1}\rangle.
 \end{align*}
Property (\ref{invg}) implies
  \begin{align*}
 \langle \operatorname d\!f(gMg^{-1})(gdBg^{-1}), \operatorname d\!h(g^{-1}Mg)(gdBg^{-1})\rangle= \langle \operatorname d\!f(M)(dB), \operatorname d\!h(M)(dB)\rangle
 \end{align*}
 \begin{align*}
 \langle \operatorname d^2\!f(gMg^{-1})(gdBg^{-1}),gdBg^{-1}\rangle
&= \langle \operatorname d^2\!f(M)(dB),dB\rangle. \qedhere
 \end{align*}
 \end{proof}
 The previous proposition implies the following one. 
\begin{prop} \label{invstable} Let $K$ be a subgroup of $\operatorname{GL}_d(\C)$. If $f$ and $h$ are elements in $\mathcal{P}(\mathfrak{g})^K$, then the functions  
\[M\in \operatorname{M}_d(\C)\mapsto  \langle\operatorname  d\!f(M)(dB),  \operatorname d\!h(M)(dB)\rangle,\]
and
\[M\in \operatorname{M}_d(\C)\mapsto   \langle \operatorname d^2\!f(M)(dB),dB\rangle\]
are also $K$-invariant polynomial functions.
\end{prop}
For a  twice continuously differentiable function $f\colon\mbox{M}_d(\C)\to \C$, multidimensional  It\^o's formula writes
 \[
df(B)=\operatorname d\! f(B)(dB)+\frac{1}{2} \langle \operatorname d^2\! f(B)(dB),dB\rangle.
 \]
 Thus   proposition \ref{invstable} leads to the next proposition  in which the integer $m$ is the one introduced in theorem \ref{invcom}.
 
  \begin{prop} \label{Markovgeneral}   If $(B(t),t\ge 0)$ is a standard  Brownian motion on $\operatorname{H}_d$, the processes 
  \[
  (\Tr(B_{i_1i_2}(t)\cdots B_{i_qi_1}(t)),t\ge 0),
  \]
 $q\in\{1,\dots,m\}$, $ i_1,\dots,i_q\in\{1,\dots,p+1\}$, form a Markov process on $\R^{r}$, with $r=\sum_{k=1}^m(p+1)^k$.
 \end{prop} 
 \begin{proof} 
 For $p$ and $q$ two integers in $\{1,\dots,m\}$ and two sequences $i_1,\dots,i_p$ and $j_1,\dots,j_q$ of integers of $\{1,\dots,p+1\}$, let us consider the functions 
$f$, $g$ and $h$ from  $\operatorname{M}_d(\C)$ to $\C$ defined by 
\[
f(M)= \Tr(M_{i_1i_2}\cdots M_{i_pi_1}),\quad g(M)= \Tr(M_{i_1i_p}\cdots M_{i_2i_1}), \quad M\in \operatorname{M}_d(\C),
\]
and
\[
 h(M)=\Tr(M_{j_1j_2}\cdots M_{j_qj_1}), \quad M\in \operatorname{M}_d(\C).
\]
Since $\overline{f(M)}=g(M)$, when $M\in \operatorname{H}_d$, we have 
 \[
  \langle\overline{\operatorname  d\!f(B)(dB)},  \operatorname d\!h(B)(dB)\rangle= \langle\operatorname  d\!g(B)(dB),  \operatorname d\!h(B)(dB)\rangle.
  \]
 Proposition (\ref{invstable}) implies that 
 \[\langle\operatorname  d\!f(B)(dB),  \operatorname d\!h(B)(dB)\rangle, \quad\langle\overline{\operatorname  d\!f(B)(dB)},  \operatorname d\!h(B)(dB)\rangle,\]
and
\[  \langle \operatorname d^2\!f(B)(dB),dB\rangle,\]
are polynomial functions in the processes 
  \[
  \Tr(B_{i_1i_2}\cdots B_{i_qi_1}),
  \]
 $q\in\{1,\dots,m\}$, $ i_1,\dots,i_q\in\{1,\dots,p+1\}$.
Thus proposition follows from usual properties of diffusions (see \cite{oksendal} for example). \end{proof}
 
Let us give a formulation of the last proposition in term of eigenvalues of some particular submatrices of Brownian motion on $\operatorname{H}_d$. In the following lemma a polynomial function 
 \[
 M \in \mbox{M}_d(\C)\mapsto f(M),\] is just denoted $f(M)$.
  \begin{lem}  \label{prodfactors} For any positive integer $q$, and any sequence of integers $i_1,\dots,i_q$ in $\{1,\dots,p+1\}$, the polynomial function 
  \[  \Tr(M_{i_1i_2}\cdots M_{i_qi_1})\] is equal to a finite product of factors of the form
  \begin{align*}
  &{\Tr}( M_{11}^n),
  {\Tr}(M_{1i}M_{i1}M_{11}^m),
  {\Tr}(M_{1i}M_{ij}M_{j1}M_{11}^n),\\
  & M_{ii}, M_{ij}M_{ji},\,( M_{ij}M_{ji})^{-1},
  M_{ij}M_{jk}M_{ki}, 
  \end{align*}
  where  $n\in \N$, and $i,j,k$, are distinct integers in $\{2,\dots,p+1\}$.
 \end{lem}
 \begin{proof} The lemma, which is  is clearly true for $q=1,2,3$, is proved by induction on $q\in\N^*$.  Suppose such a decomposition exists up to $q-1$, for  a fixed integer $q$ greater than $4$. Let us consider a sequence of integers  $i_1,\dots,i_q$ in $\{1,\dots,p+1\}$. If all the integers of the sequence or none of them are equal to $1$, then the decomposition exists. If it exists two successive  integers, say $i_1$ and $i_2$, such that $i_1=1$, $i_2\ne 1$  then  it exists  integers  $k\le q-1$ and $p\le q$,  such that  $ i_p\ne 1$, and 
 \[ \Tr(M_{i_1i_2}\cdots M_{i_qi_1})=M_{i_2i_3}\cdots M_{i_p1}M^k_{11}M_{1i_2}.\]
 If $i_p=i_2$, then 
  \[ 
  \Tr(M_{i_1i_2}\cdots M_{i_qi_1})=(M_{i_2i_3}\cdots M_{i_{p-1}i_2})(M_{i_21}M^k_{11}M_{1i_2}).
  \]
   If $i_p\ne i_2$, then 
  \[
   \Tr(M_{i_1i_2}\cdots M_{i_qi_1})(M_{i_2i_p}M_{i_pi_2})=(M_{i_2i_3}\cdots M_{i_{p-1}i_p}M_{i_pi_2})(M_{i_p1}M^k_{11}M_{1i_2}M_{i_2i_p}).
   \]
 Induction hypothesis implies that  the above polynomials can be written as a product of factors given in the lemma. 
 \end{proof}
   \begin{prop}  \label{proptr}  If $B$ is a Brownian motion on $\operatorname{H}_d$, then the processes
  \begin{align*}
  &{\Tr}( B_{11}^n),
  {\Tr}(B_{1i}B_{i1}B_{11}^m),
  {\Tr}(B_{1i}B_{ij}B_{j1}B_{11}^n),\\
  & B_{ii}, B_{ij}B_{ji},
  B_{ij}B_{jk}B_{ki}, 
  \end{align*}
  where  $n\in \N$, and $i,j,k$, are distinct integers in $\{2,\dots,p+1\}$, taken together, form a Markov process.
 \end{prop}
 \begin{proof} 
Lemma \ref{prodfactors} implies that there is a bijection between the Markov process of   proposition \ref{Markovgeneral} and the process of   proposition \ref{proptr}, which is consequently Markovian too.
 \end{proof}
The following theorem is an immediate consequence of  the previous proposition.
 \begin{thm} \label{maintheomat} Let $p$ be a positive integer and $B$ be a Brownian motion on $\operatorname{H}_d$. Then the processes of the eigenvalues of the matrices, 
    \begin{align*}
  &B_{11},\,
\left(
\begin{array}{cc}
 B_{11} & B_{1i} \\
 B_{i1} &  B_{ii}   
\end{array}
\right),\, \left(
\begin{array}{ccc}
 B_{11} & B_{1i} & 0\\
0 & 0 &B_{ij}  \\
 B_{j1} & 0 & 0\\
\end{array}
\right),
  \end{align*}
  and the complex  processes,
 \[ B_{ij}B_{ji},\, B_{ij}B_{jk}B_{ki}, \]
 where $i,j,k$, are distinct integers in $\{2,\dots,p+1\}$, taken together, form a Markov process.
 \end{thm}
  Taking $p=1$ in theorem \ref{maintheomat} we obtain the following corollary, which has been already proved in \cite{Nordenstam}.
 \begin{cor} If $(\Lambda^{(d)}(t),t\ge 0)$ is the process  of eigenvalues of a standard Brownian motion on $\operatorname{H}_d$   and $(\Lambda^{(d-1)}(t),t\ge 0)$ is the process  of eigenvalues of its principal minor of order $d-1$, then the processes 
 \[
 (\Lambda^{(d)}(t),\Lambda^{(d-1)}(t),t\ge 0)
 \]
  is Markovian.
 \end{cor}

\end{document}